\documentclass[11pt, a4paper]{article}

\usepackage[margin=1.2in]{geometry}

\usepackage[utf8]{inputenc}
\usepackage{comment}
\usepackage{amsmath}
\usepackage[USenglish]{babel}
\usepackage{hyperref}
\usepackage{amsthm}
\usepackage{enumerate}
\usepackage{relsize}
\usepackage{tikz}
\usepackage{graphicx}
\usetikzlibrary{arrows,shapes}
\usepackage{amsfonts}
\hypersetup{
            colorlinks=true,
            linkcolor=blue,
            citecolor=blue
            }

\newtheorem{theorem}{Theorem}[section]
\newtheorem{corollary}[theorem]{Corollary}
\newtheorem{lemma}[theorem]{Lemma}
\newtheorem{example}[theorem]{Example}
\newtheorem{proposition}[theorem]{Proposition}

\numberwithin{equation}{section}

 \title{Supercharacters for Normal Supercharacter Theory}
 \author{Farid Aliniaeifard}
 \date{}

\begin{document}

\maketitle

\newcommand{\cX}{\mathcal{X}}
\newcommand{\cK}{\mathcal{K}}
\newcommand{\Irr}{\textrm{Irr}}
\newcommand{\one}{1\hspace{-.12cm}1}
\newcommand{\Sup}{\mathrm{Sup}}


\begin{abstract}
{  In order to find a tractable theory to substitute for the wild character theory of the group of $n\times n$ unipotent upper-triangular matrices over a finite field $\mathbb{F}_q$, Andr\'e and Yan introduced the  notion of supercharacter theory.  In this paper, we construct a supercharacter theory from an arbitrary set $S$ of normal subgroups of $G$. We call such supercharacter theory the normal supercharacter theory generated by $S$. It is shown that normal supercharacter theories are integral, and a recursive formula for supercharacters of the normal supercharacter theory is provided. The normal supercharacter theory provides  many substitutions for wild character theories.  Also, we indicate that the superclasses of the normal supercharacter theory generated by all normal subgroups of $G$ are given by certain values on the primitive central idempotents. We study the connection between the finest normal supercharacter theory and faithful irreducible characters. Moreover, an algorithm is presented to construct the supercharacter table of the finest normal supercharacter theory from the character table. Finally, We justify that normal supercharacter theories cannot be obtained by preceding supercharacter theory constructions.
\vspace{3mm}
}
\end{abstract}

\tableofcontents

\section{Introduction}

Let $UT_n(q)$ denote the group of $n \times n$ unipotent upper-triangular matrices over a finite field $\mathbb{F}_q$. Classification of the irreducible characters of $UT_n(q)$ is a well-known wild problem, provably intractable for arbitrary $n$. In order to find a more tractable way to understand the representation theory of $UT_n(q)$, C. Andr\'e \cite{And95} defined and constructed supercharacter theory.  Yan \cite{Yan01} showed  how to  replace Andr\'e's construction with more elementary methods.  Diaconis and Isaacs \cite{DI08} axiomatized the concept of supercharacter theory for an arbitrary group. They mentioned two supercharacter theory constructions for $G$, one comes from the action of a group $A$ on $G$ by automorphisms and another one comes from the action of cyclotomic field $Q[\zeta_{|G|}]$. They also generalized Andr\'e's original construction to define a supercharacter theory for \textbf{\emph{algebra group}}, a group of the form $1+J$ where $J$ is a finite dimensional nilpotent associative algebra over a finite field $\mathbb{F}$ of characteristic $p$.  Later, in \cite{Hen12,Hen09}, Hendrickson provided other supercharacter theories for $G$. Arias-Castro, Diaconis, and Stanley \cite{ADS04} used Yan’s work in place of the usual
irreducible character theory to study random walks on $UT_n(q)$.
Also in \cite{AABB12}, the authors  obtained a relationship
between the supercharacter theory of all unipotent upper-triangular matrices over a finite field $\mathbb{F}_q$ simultaneously and
the combinatorial Hopf algebra of symmetric functions in non-commuting variables.\\

Let $N(G)$ be the set of all normal subgroups of $G$. The product of two normal subgroups is a normal subgroup. Therefore, $N(G)$ is a semigroup. Let $S\subseteq N(G)$.  We define
$A(S)$ to be the smallest subsemigroup of $N(G)$ such that
\begin{enumerate}
\item $\{1\},G\in A(S)$.
\item $S\subseteq A(S).$
\item  $A(S)$ is closed under intersection.
\end{enumerate}

 Note that every element $N\in A(S)$ is a normal subgroup of $G$. We define for an element $N\in A(S)$,
 $${N^{\circ}_{A(S)}}:=N\setminus \bigcup_{H\in A(S), H\subset N} H.$$
For simplicity of notation, we write ${N^{\circ}}$ instead of ${N^{\circ}_{A(S)}}$ when it is clear that $N$ is in $A(S)$.

A subgroup of $G$ is normal if and only if it is the union of a set of conjugacy classes of $G$.  We have an equivalent characterization of normality in terms of the kernels of irreducible characters.  Recall that the kernel of a character $\chi$ of $G$ is the set $\ker \chi = \{ g \in G : \chi(g) = \chi(1)\}$.  This is just the kernel of any representation whose character is $\chi$, and so $\ker \chi$ is a normal subgroup of $G$.  A subgroup of $G$ is normal if and only if it is the intersection of the kernels of some finite set of irreducible characters \cite[Proposition 17.5]{JL93}; thus the normal subgroups of $G$ are the subgroups which we can construct from the character table of $G$.

 Let $N$ be a normal subgroup of $G$. If $\psi$ is an irreducible character of $G$, then define $\overline{\psi}: G/N \to \mathbb{C}$ by  $\overline{\psi}(gN)=\psi(g)$. Note that $\overline{\psi}$ is an irreducible character of $G/N$ and we say $\overline{\psi}$ has been lifted to the irreducible character $\psi$ of $G$. If $\theta$ is an irreducible of $G/N$, then define  $\widetilde{\theta}: G \to \mathbb{C}$ by $\widetilde{\theta}(g)=\theta(gN)$. Notice that $\widetilde{\theta}$ is an irreducible character of $G$ and we say $\widetilde{\theta}$ has been lifted from $\theta$. We have the following bijection between irreducible characters $\chi$ of $G$ containing $N$ in their kernels and irreducible characters of $G/N$,
$$
\begin{array}{cccc}
& \{\psi \in {\rm Irr}(G): N\subseteq {\rm ker}~\psi \}& \to&  {\rm Irr}(G/N)\\

       & \psi  &\mapsto &\overline{\psi}.
\end{array}
$$

Let for each $N\in A(S)$,
$$\cX^N:=\{\psi \in  \Irr(G): N\subseteq \ker \psi\}$$
and
$$\chi^N:=\sum_{\psi \in \cX^N} \psi(1)\psi.$$
Indeed, $\cX^{N}$ is the set of irreducible characters of $G$ lifting from irreducible characters of $G/N$. Therefore,
$$\chi^N(g)=\rho_{G/N}(gN),$$
where $\rho_{G/N}$ is the regular character of $G/N$.
Define for every $N\in A(S)$,

$$\cX_{A(S)}^{N^\bullet}:=\cX^N\setminus \bigcup_{\substack{
    K\in A(S):\\
   N\subset K}
  } \cX^K$$
  and
  $$\chi_{A(S)}^{{N}^\bullet}:=\sum_{\psi \in \cX_{A(S)}^{N^\bullet}} \psi(1)\psi.$$
We may assume that $\chi_{A(S)}^{{N}^\bullet}=0$, whenever  $\cX_{A(S)}^{N^\bullet}=\emptyset$. For simplicity of notation, we write $\cX^{N^{\bullet}}$ and $\chi^{N^{\bullet}}$ instead of $\cX_{A(S)}^{N^\bullet}$ and  $\chi_{A(S)}^{{N}^\bullet}$ respectively when it is clear that $N$ is in $A(S)$.

 A character $\chi$ of a group $G$ is said to be \textbf{\emph{integral}} if $\chi(g)\in \mathbb{Q}$ for every element $g\in G$; a supercharacter theory is said to be \textbf{\emph{integral}} if its supercharacters are integral. In Theorem \ref{normal}, we will show that for an arbitrary subset $S\subseteq N(G)$,
$$\left(\{\cX^{N^{\bullet}}\neq \emptyset: N\in A(S)\},\{N^\circ \neq \emptyset: N\in A(S)\}\right)$$
is a supercharacter theory of $G$. We call such supercharacter theory the \textbf{\emph{normal supercharacter theory}} generated by $S$.
Note that when we have a larger set of normal subgroups, the normal supercharacter theory we obtain will be finer. In particular the finest normal supercharacter theory is obtained when we consider the set of all normal subgroups of $G$, and is related to a partition of $G$ given by certain values on the primitive central idempotents. The lattice of normal subgroups has been well studied (for example see \cite{Bae38,Bir67,Jon54}); we show that every sublattice of the lattice of normal subgroups of $G$ containing the trivial subgroup and $G$ yields a normal supercharacter theory. We also identify supercharacters for any normal supercharacter theory and we show that every normal supercharacter theory is integral.\\

In Section \ref{background}, we review definitions and notations for supercharacter theories, and we mention the known supercharacter theory constructions. In Section \ref{main}, we show that the normal supercharacter theory generated by an arbitrary subset $S\subseteq N(G)$ is integral. Furthermore, we provide a recursive formula for the supercharacters of any normal supercharacter theory. In Section \ref{finest}, we study the finest normal supercharacter theory and we show that the finest normal supercharacter theory is obtained by considering certain values of the primitive central idempotents. Also, It is shown that $\chi^{N^\bullet}$ is equal to the set of all faithful irreducible characters of $G/N$. In the end of this section, we provide an algorithm to  construct the finest normal supercharacter theory from the character table. In section \ref{new}, we justify that the normal supercharacter theory  cannot be obtained by the preceding main constructions. 

\section{Background}\label{background}
In this section, we first review the definitions for supercharacter theories, then we discuss the main methods for obtaining supercharacter theories. In the end, we mention the M{\"o}bius inversion formula and its dual, which we will frequently use them in the sequel.

\subsection{Supercharacter Theory}\label{supercharacter theory}
We reproduce the definition of supercharacter theory in \cite{DI08}. Throughout this paper for every subset $X\subseteq \Irr(G)$ let $\sigma_X$ be the character $\sum_{\psi\in X}\psi(1)\psi$.

A \textbf{\emph{supercharacter theory}} of a finite group $G$ is a pair $(\cX,\cK)$ with a choice of a set of characters $\{\chi_X:X\in \cX\}$, where
\begin{itemize}
  \item $\cX$ is a partition of $\Irr(G)$ and
  \item  $\cK$ is a partition of $G$,
\end{itemize}
such that:
\begin{enumerate}
  \item The set $\{1\}$ is a member of $\cK$.
  \item $|\cX|=|\cK|$.
  \item The characters $\chi_X$ are constant on the parts of $\cK$.
\end{enumerate}
We refer to the characters $\chi_X$ as \textbf{\emph{supercharacters}} and to the members of $\cK$ as \textbf{\emph{superclasses}}. It is well-known that every supercharacter $\chi_X$ is a constant multiple of $\sigma_{X}$ (See \cite[Section 2]{DI08}). Denote by $\Sup(G)$ the set of all supercharacter theories of $G$.

Every finite group $G$ has two trivial supercharacter theories: the usual irreducible character theory
and the supercharacter theory $(\{\{\one\},\Irr(G)\setminus\{\one\}\},\{ \{1\}, G \setminus \{1\}\} )$, where
$\one$ is the principal character of $G$. In the following subsections we review known supercharacter theory constructions for $G$.

\subsection{A Group Acts Via Automorphisms on a Given Group }\label{act via automorphisms}
Given finite groups $A$ and $G$, we say that $A$ acts via automorphisms  on $G$ if $A$ acts on $G$ as a set, and in addition $a.(gh)=(a.g)(a.h)$ for all $g,h\in G$ and $a\in A$. An action via automorphisms of $A$ on $G$ determines and is determined by a homomorphism $\phi: A\rightarrow Aut(G)$.

Suppose that $A$ is a group that acts via automorphisms on our given group $G$. It is well known that
$A$ permutes both the irreducible characters of $G$ and the conjugacy classes of $G$. By a lemma of R. Brauer, the permutation characters of $A$ corresponding to these two actions
are identical, and so the numbers of $A$-orbits on ${\rm Irr}(G)$ and on the set of
classes of $G$ are equal (See Theorem 6.32 and Corollary 6.33 of \cite{Isa94}). It is easy to see that these orbit
decompositions yield a supercharacter theory $(\cX,\cK)$ where the members of $\cX$ are the $A$-orbits on the classes of $G$ and members of $\cK$ are the unions of the $A$-orbits on the classes of $G$. It is
clear that in this situation, the sum of the characters in an orbit $X\in \cX$ is constant on
each member of $\mathcal{K}$. We denote by ${\rm AutSup}(G)$ the set of all such supercharacter theories of $G$.

\subsection{Action of Automorphisms of The Cyclotomic Field $\mathbb{Q}[{\zeta}_{|G|}]$ }\label{action of automorphisms of cyclotomic field}
Another general way to construct a supercharacter theory for $G$ uses the action of a
group $A$ of automorphisms of the cyclotomic field $ Q{[{\zeta}_{|G|}]}$, where ${\zeta}_{|G|}$ is a primitive $|G|$th
root of unity. Clearly, $A$ permutes $\Irr(G)$, and there is a compatible action on the classes of $G$, defined as follows. Given $\sigma \in A,$ there is a unique positive integer $r < |G|$ such that
$\sigma({\zeta}_{|G|}) = {\zeta}_{|G|}^{r}$, and we let $\sigma$ carry the class of $g \in G $ to the class of $g^{r}$. The Brauer's lemma shows that the numbers of $A$-orbits on $\Irr(G)$ and on the set of classes of $G$ are equal.
In this case too, we take $\cX$ to be the set of $A$-orbits on $\Irr(G)$, and again, $\cK$ is the set of
unions of the various $A$-orbits on conjugacy classes. As before, it is trivial to check that
the sum of the characters in an orbit $X\in \cX$ is constant on each member of $\cK$, and so, we have a supercharacter theory.
We denote by ${\rm ACSup}(G)$ the set containing above supercharacter theories.

\subsection{The $\ast$-Product }\label{star product}

In this subsection, we show that if $N$ is a normal subgroup of $G$, then some supercharacter theories of $N$ can be combined with supercharacter theories of $G/N$ to form supercharacter theories of the full group $G$.

Let $G$ and $H$ be groups and let $G$ act on $H$ by automorphisms. We say that $(\cX,\cK)\in \Sup(H)$ is $G$-invariant if the action of $G$ fixes each part $K\in \cK$ setwise. We denote by $\Sup_G(H)$ the set of all $G$-invariant supercharacter theories of $H$. \\

For each subset $L\subseteq G/N$ let $\widetilde{L}:=\bigcup_{Ng\in L}Ng$. Extend this notation to a set $\mathcal{L}$ of subsets of $G/N$ by $\widetilde{\mathcal{L}}:=\{\widetilde{L}: L \in \mathcal{L}\}$, and let $\mathcal{L}^\circ$ denote $\mathcal{L}\setminus\{\{N\}\}$.

Recall that if $N$ is a normal subgroup of $G$ and $\psi\in \Irr(N)$, then $\Irr(G|N)$ denotes the set of irreducible characters $\chi$ of $G$ such that $\langle \chi_{\downarrow_N},\psi\rangle>0$, where $\chi_{\downarrow_N}$ is the restriction of $\chi$ to $N$. If $Z\subseteq \Irr(N)$ is a union of $G$-orbits, then define the subset $Z^G$ of $\Irr(G)$ to be $\bigcup_{\psi\in Z}\Irr(G|\psi)$. Extend this notation to a set $\mathcal{Z}$ of subsets of $\Irr(N)$ by letting $\mathcal{Z}^G:=\{Z^G:Z\in \mathcal{Z}\}$, and let $\mathcal{Z}^\circ:=\mathcal{Z}\setminus \{\{\one_N\}\}$.

Define
$$
\begin{array}{cccc}
\ast_N:& \Sup_G(N)\times \Sup(G/N)                         & \to     & \Sup(G)\\

     & \left((\cX,\cK),(\mathcal{Y},\mathcal{L})\right)  & \mapsto & \left( (\cX^\circ)^G\cup \mathcal{Y}, \cK\cup \widetilde{\mathcal{L}^\circ} \right),
\end{array}
$$
then by \cite[Theorem 5.3]{Hen12} the map $*$ is well-defined. Denoted by $\Sup^\ast_N(G)$ the image of $\ast_N$ and let $\Sup^\ast(G)=\bigcup_{N\in N(G)}\Sup^\ast_N(G)$.

\subsection{The M{\"o}bius Inversion Formula and Its Dual}\label{mobius}
We mention the M{\"o}bius inversion formula and its dual which we will frequently use them in the sequel.

If $(P,\leq)$ is a poset and $\Bbb{C}^{P\times P}$ is the set of all functions $\alpha: P\times P \rightarrow \Bbb{C}$, the associated incidence algebra is
$$A(P) = \{\alpha \in \Bbb{C}^{P \times P}
: \alpha(s, u) = 0 ~{\rm unless}~ s \leq u\}.
$$
The \textbf{\emph{M{\"o}bius function}} $\mu\in A(P)$ is defined recursively  by the following rules:
 $$\mu(s, s) = 1,$$ and
$$\mu(s, u) =- \sum_{s\leq  t< u}{\mu(t,u)}, ~~\text{for all} ~s<u~\text{in} ~P.$$
It is immediate from this definition that
\[  \sum_{s\leq t\leq u}{\mu(t,u)}= \left\{
  \begin{array}{l l}
    1 & \quad \text{if $s=u$, }\\
    0 & \quad \text{otherwise.}
  \end{array}
	\right.\]
	
\begin{lemma}\cite[Proposition 3.7.1 (M{\"o}bius inversion formula)]{Sta86}
Let $P$ be a finite poset. Let $f,g:P\rightarrow \mathbb{K}$, where $\mathbb{K}$ is a field. Then
$$g(t)=\sum_{s\leq t} f(s), ~~\text{for all}~t\in P,$$
if and only if
$$f(t)=\sum_{s\leq t} g(s)\mu(s,t), ~~\text{for all}~t\in P.$$
\end{lemma}

\begin{lemma}\cite[Proposition 3.7.2 (M{\"o}bius inversion formula, dual form)]{Sta86}
Let $P$ be a finite poset. Let $f,g:P\rightarrow \mathbb{K}$, where $\mathbb{K}$ is a field. Then
$$g(t)=\sum_{s\geq t} f(s), ~~\text{for all}~t\in P,$$
if and only if
$$f(t)=\sum_{s\geq t} g(s)\mu(t,s), ~~\text{for all}~t\in P.$$
\end{lemma}

\section{Normal Supercharacter Theory}\label{main}
In this section we show that our constructions are indeed supercharacter theories of $G$. Furthermore, we show that these supercharacter theories are integral.

A \textbf{\emph{lattice}} is a partial ordered set in which every two elements have a unique least upper bound or \textbf{\emph{join}}  and a unique greatest lower bound or \textbf{\emph{meet}}.
\begin{proposition}\label{lattice}
Let $G$ be a group. Then for an arbitrary subset $S \subseteq N(G)$, $A(S)$ is a sublattice of the lattice of normal subgroups of $G$ containing the trivial subgroup and $G$.
\end{proposition}

\begin{proof}
It is easy to see that $A(S)$ is a poset. Let $N,H\in A(S)$. Then the least upper bound of $N$ and $H$ exists and that is $NH\in A(S)$; also the the greater lower bound of $N$ and $H$ exists and that is $N \cap H\in A(S)$. Therefore, $A(S)$ is a sublattice of the lattice of normal subgroups of $G$ containing the trivial subgroup and $G$.
\end{proof}

\begin{example}
The Hasse diagram of $A(S)$, where $S$ is the set of all normal pattern subgroups of $UT_4(q)$.

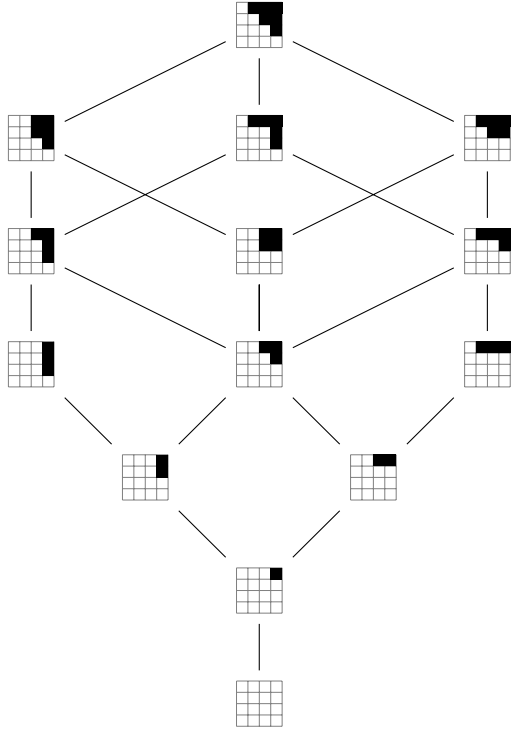
\begin{center}
\begin{tikzpicture}[scale=1, auto,swap]
\node (1) at (0,0){$\begin{tikzpicture}[scale=.15]
                     \draw[step=1cm,gray,very thin] (0,0) grid (4,4);
                     \end{tikzpicture}$};
\node (2) at (0,1.5){$\begin{tikzpicture}[scale=.15]
                   \draw[step=1cm,gray,very thin] (0,0) grid (4,4);\fill[black] (3,3) rectangle (4,4);
                   \end{tikzpicture} $};
\node (3) at (1.5,3){$\begin{tikzpicture}[scale=.15]
                   \draw[step=1cm,gray,very thin] (0,0) grid (4,4);
                   \fill[black] (2,3) rectangle (4,4);
                   \end{tikzpicture}$};
\node (4) at (-1.5,3){$\begin{tikzpicture}[scale=.15]
                   \draw[step=1cm,gray,very thin] (0,0) grid (4,4);
                   \fill[black] (3,2) rectangle (4,4);
                    \end{tikzpicture}$};
\node (5) at (3,4.5){$\begin{tikzpicture}[scale=.15]
                    \draw[step=1cm,gray,very thin] (0,0) grid (4,4);
                    \fill[black] (1,3) rectangle (4,4);
                    \end{tikzpicture}$};
\node (6) at (0,4.5){$\begin{tikzpicture}[scale=.15]
                     \draw[step=1cm,gray,very thin] (0,0) grid (4,4);
                     \fill[black] (3,2) rectangle (4,4); \fill[black] (2,3) rectangle (4,4);
                     \end{tikzpicture}$};
\node (7) at (-3,4.5){$\begin{tikzpicture}[scale=.15]
                    \draw[step=1cm,gray,very thin] (0,0) grid (4,4);
                    \fill[black] (3,1) rectangle (4,4);
                    \end{tikzpicture}$};
\node (8) at (3,6){$\begin{tikzpicture}[scale=.15]
                    \draw[step=1cm,gray,very thin] (0,0) grid (4,4);
                    \fill[black] (3,2) rectangle (4,4); \fill[black] (1,3) rectangle (4,4);
                    \end{tikzpicture}$};
\node (9) at (0,6){$\begin{tikzpicture}[scale=.15]
                     \draw[step=1cm,gray,very thin] (0,0) grid (4,4);
                     \fill[black] (2,2) rectangle (4,4);
                     \end{tikzpicture}$};
\node (10) at (-3,6){$\begin{tikzpicture}[scale=.15]
                      \draw[step=1cm,gray,very thin] (0,0) grid (4,4);
                      \fill[black] (2,3) rectangle (4,4); \fill[black] (3,1) rectangle (4,4);
                      \end{tikzpicture}$};
\node (11) at (3,7.5){$\begin{tikzpicture}[scale=.15]
                       \draw[step=1cm,gray,very thin](0,0) grid (4,4);
                       \fill[black] (1,3) rectangle (4,4); \fill[black] (2,2) rectangle (4,4);
                       \end{tikzpicture}$};
\node (12) at (0,7.5){$\begin{tikzpicture}[scale=.15]
                      \draw[step=1cm,gray,very thin] (0,0) grid (4,4);
                      \fill[black] (1,3) rectangle (4,4); \fill[black] (3,1) rectangle (4,4);
                      \end{tikzpicture}$};
\node (13) at (-3,7.5){$\begin{tikzpicture}[scale=.15]
                       \draw[step=1cm,gray,very thin] (0,0) grid (4,4);
                       \fill[black] (2,2) rectangle (4,4); \fill[black] (3,1) rectangle (4,4);
                       \end{tikzpicture}$};
\node (14) at (0,9){$\begin{tikzpicture}[scale=.15]
                      \draw[step=1cm,gray,very thin](0,0) grid (4,4);
                      \fill[black] (2,2) rectangle (4,4); \fill[black] (3,1) rectangle (4,4); \fill[black] (1,3) rectangle (4,4);
                      \end{tikzpicture}$};

\draw (1)--(2)--(3)(2)--(4)--(7)(4)--(6)--(3)--(5)--(8)--(6)--(9)--(6)--(10)--(7)(10)--(13)--(9)--(11)(8)--(11)--(14)--(12)(14)--(13)(8)--(12)--(10);
\end{tikzpicture}
\end{center}

\end{example}

\begin{lemma}\label{partitions}
Let $G$ be a group. Then the following are hold for an arbitrary subset $S\subseteq N(G)$.
\begin{enumerate}
\item $\left\lbrace \cX^{N^{\bullet}}\neq \emptyset: N\in A(S)\right\rbrace$ is a partition of $\Irr(G)$.
\item $\left\lbrace N^\circ \neq \emptyset: N\in A(S) \right\rbrace $ is a partition of $G$.
\end{enumerate}
\end{lemma}

\begin{proof}
Recall that for every $N\in A(S)$,
$${N^\circ}=N\setminus \bigcup_{H\in A(S), H\subset N} H.$$
Therefore, $g\in N^\circ$ if and only if $N$ is the smallest normal subgroup in $A(S)$ containing $g$. This implies that $N^\circ\cap K^\circ=\emptyset$ for every normal subgroups $N,K\in A(S)$. Also, since $G\in A(S)$, for every $g\in G$ there exists a normal subgroup $N\in A(S)$ such that $g\in N^\circ$. This completes the proof of (1).

Recall that for every $N\in A(S)$,
$\cX^N:=\{\psi \in  \Irr(G): N\subseteq \ker \psi\}$ and
$$\cX^{N^\bullet}:=\cX^N\setminus \bigcup_{\substack{
    K\in A(S):\\
   N\subset K}
  } \cX^K.$$
Therefore, $\cX^{N^\bullet}$ is the set of irreducible characters $\psi$ of $G$ such that if for a normal subgroup $H\in A(S)$, $H\subseteq \ker \psi$, then $H\subseteq N$. This implies that $\cX^{N^\bullet}\cap \cX^{K^{\bullet}}=\emptyset$ for every normal subgroups $N,K\in A(S)$.  Also, since $\cX^{\{1\}}=\Irr(G)$, for every $\psi \in \Irr(G)$ there exists a normal subgroup $N\in A(S)$ such that $\psi\in \cX^{N^\bullet}$. This completes the proof of (2).

\end{proof}

\begin{theorem}\label{normal}
Let $G$ be a group. Then for an arbitrary subset $S\subseteq N(G)$,
$$\left(\{\cX^{N^{\bullet}}\neq \emptyset: N\in A(S)\},\{N^\circ \neq \emptyset: N\in A(S)\}\right)$$
is an integral supercharacter theory of $G$. Moreover,
$$\chi^{N^\bullet}(g)=\sum_{
                         \substack
                                 {
                                  M\in A(S):\\
                                  g\in M, N\subseteq M
                                 }
                        }
                        \mu(N,M)\frac{|G|}{|M|}.$$
\end{theorem}

\begin{proof}
Let $\cX=\{\cX^{N^{\bullet}}\neq \emptyset: N\in A(S)\}$ and $\cK=\{N^\circ \neq \emptyset: N\in A(S)\}$. By Lemma \ref{partitions}, $\cX$ is a partition of $\Irr(G)$ and $\cK$ is a partition of $G$. Recall that  $$\chi^{{N}^\bullet}:=\sum_{\psi \in \cX_{A(S)}^{N^\bullet}} \psi(1)\psi.$$
We aim to show that $(\cX,\cK)$ is an integral supercharacter theory of $G$ with the choose of the set of characters $\{\chi^{N^\bullet}\neq 0: N\in A(S)\}$. Therefore, we only need to show the following are hold.
\begin{enumerate}
\item The set $\{1\}$ is a member of $\cK$.
\item $|\cX|=|\cK|$.
\item The characters $\chi^{N^\bullet}$ are integral and constant on the parts of $\cK$.
\end{enumerate}

Since $\{1\}=\{1\}^\circ$, $\{1\}$ is a member of $\cK$. Therefore, (1) is hold.

Define for every $N\in A(S)$,
$$f^{N^{\circ}}(g):=
          \begin{cases}
            1       & \quad \text{if } g\in N^\circ, \\
            0  & \quad  \text{ Otherwise}.\\
          \end{cases}$$
Thus,
$$\sum_{
        \substack{
                  N\in A(S):\\
                  N\subseteq K
                  }
       }
        f^{N^\circ}(g)=
        \begin{cases}
                     1  & \quad  \text{if } g\in K, \\
                     0  & \quad  \text{ Otherwise}.\\
        \end{cases}$$
Recall that
$$\chi^K(g)=\rho_{G/K}(gK)=\begin{cases}
                                   \frac{|G|}{|K|}  & \quad  \text{if } g\in K, \\
                                   0                & \quad  \text{ Otherwise}.\\
                       \end{cases}$$
Therefore,
          $$\chi^{K}(g)=\frac{|G|}{|K|}\sum_{
                                             \substack{
                                                       N\in A(S):\\
                                                       N\subseteq K
                                                      }
                                            }f^{N^\circ}(g).$$
It follows from M{\"o}bius inversion formula that
$$f^{K^\circ}(g)=\sum_{
                       \substack
                                {
                                N\in A(S):\\
                                N\subseteq K
                                }
                       }
                        \frac{|N|}{|G|} \mu(N,K)\chi^N(g).$$
Since
$$\chi^{N}=\sum_{
                 \substack
                         {
                          M\in A(S):\\
                          N\subseteq M
                          }
                }
                \chi^{M^\bullet},$$
We conclude that
$$f^{K^\circ}=\sum_{
                     \substack
                              {
                               N\in A(S):\\
                               N\subseteq K
                               }
                    }
                     \left( \frac{|N|}{|G|}\sum_{
                                                \substack
                                                        {
                                                         M\in A(S):\\
                                                         N\subseteq M
                                                         }
                                               }
                                               \mu(N,K)\chi^{M^\bullet}\right).$$
Therefore,
$$\mathbb{C}\text{-Span}\{\chi^{N^\bullet}: N\in A(S) \}=\mathbb{C}\text{-Span}\{f^{N^\circ}: N\in A(S) \}.$$
Note that
$$|\mathcal{X}|={\rm dim} \left( \mathbb{C}\text{-Span}\{\chi^{N^\bullet}: N\in A(S) \} \right) ={\rm dim}\left( \mathbb{C}\text{-Span}\{f^{N^\circ}: N\in A(S) \}\right)=|\mathcal{K}|.$$
We can see that (2) is hold.

Since
$$\chi^{N}=\sum_{
                 \substack
                         {
                          M\in A(S):\\
                          N\subseteq M
                          }
                }
                \chi^{M^\bullet},$$
by the dual of M{\"o}bius inversion formula we have
$$\chi^{N^\bullet}=\sum_{
                         \substack
                                 {
                                  M\in A(S):\\
                                  N\subseteq M
                                 }
                        }
                        \mu(N,M)\chi^{M}.$$
Note that
$$\chi^M(g)=\rho_{G/M}(gM)=\begin{cases}
                                   \frac{|G|}{|M|}  & \quad  \text{if } g\in M, \\
                                   0                & \quad  \text{ Otherwise}.\\
                       \end{cases}$$
It follows that
$$\chi^{N^\bullet}(g)=\sum_{
                         \substack
                                 {
                                  M\in A(S):\\
                                  g\in M, N\subseteq M
                                 }
                        }
                        \mu(N,M)\frac{|G|}{|M|}.$$
Let $h,h_1\in H^\circ \in \cK$. Then for every normal subgroup $M$ of $G$, $h\in M$ if and only if $H^\circ \subseteq M$ if and only if $h_1\in M$. So $\chi^{N^\bullet}(h)=\chi^{N^\bullet}(h_1)$. Thus, $\chi^{N^\bullet}$ is integral and constant on each part of $\cK$. Therefore, (3) is hold.
\end{proof}

In the following theorem we give a recursive formula for supercharacter $\chi^{N^\bullet}$ of the normal supercharacter theory generated by $S\subseteq N(G)$.

\begin{theorem}\label{recursive}
Let $G$ be a group and let $S$ be an arbitrary subset of $N(G)$. Then for any normal subgroup $N\in A(S)$
$$
\chi^{N^\bullet}(g) =
                     \begin{cases}
                                  \mathlarger{\sum}_{
                                                     \psi\in \mathcal{X}^{N^\bullet}
                                                    }\psi(1)^2  & g\in N,\\
                                                    ~\\
                                  -\mathlarger{\sum}_{
                                                      \substack{
                                                                K\in A(S):\\ N\subset K
                                                               }
                                                      }\chi^{K^\bullet}(g)  & g\not\in N.\\
                     \end{cases}
$$
\end{theorem}

\begin{proof}
Let $N\in A(S)$. Consider the following two cases.

{\it Case 1}. $g\in N$. We have
$$\chi^{N^\bullet}(g)=\sum_{\psi\in \mathcal{X}^{N^\bullet}}\psi(1)\psi(g).$$
For every $\psi\in \mathcal{X}^{N^\bullet}$, we see that $N\subseteq {\rm ker}~\psi$. Thus, $\psi(g)=\psi(1)$. Therefore,
 $$\chi^{N^\bullet}(g)=\sum_{\psi\in \mathcal{X}^{N^\bullet}}\psi(1)^2.$$

{\it Case 2}. $g\not\in G$. Note that
$${\rm Irr}\left(G/N\right)=\{ \overline{\psi}: \psi \in \mathcal{X}^{N}\}=
\bigcup_{
         \substack{
                   K\in A(S):\\
                   N\subseteq K
                  }
         } \{\overline{\psi}: \psi \in \mathcal{X}^{K^\bullet}\}.$$ Therefore,

$$\rho_{G/N}=\sum_{
                   \substack{
                             \psi \in \mathcal{X}^{{K}^\bullet}:\\
                             N\subseteq K\in A(S)
                            }
                  }\overline{\psi}(1) \overline{\psi}.$$
Thus,
$$0=\rho_{G/N}(gN)=
\sum_{
     \substack{
              \psi \in \mathcal{X}^{K^\bullet}:\\
              N\subseteq K\in A(S)
              }
     }\overline{\psi}(1) \overline{\psi}(gN)=$$
$$\sum_{
        \substack{
                 \psi \in \mathcal{X}^{K^\bullet}:\\
                 N\subset K\in A(S)
                 }
       }
       \overline{\psi}(1) \overline{\psi}(gN)+\sum_{\psi\in \mathcal{X}^{N^\bullet}}\overline{\psi}(1)\overline{\psi}(gN)=$$
$$\sum_{
        \substack{
                 \psi \in \mathcal{X}^{K^\bullet}:\\
                 N\subset K\in A(S)
                 }
       }
       \psi(1) \psi(g)+\sum_{\psi\in \mathcal{X}^{N^\bullet}}\psi(1)\psi(g)=$$
                $$\sum_{\substack{
                                 \psi \in \mathcal{X}^{K^\bullet}:\\
                                   N\subset K\in A(S)
                                 }
                       }\psi(1) \psi(g)+\chi^{N^\bullet}(g).$$
It follows that

$$\chi^{{N^\bullet}}(g)=-\sum_{
                               \substack
                                        {
                                         \psi \in \mathcal{X}^{K^\bullet}:\\
                                          N\subset K\in A(S)
                                        }
                              }\psi(1) \psi(g).$$
Also,
$$\sum_{
       \substack
                {
                 \psi \in \mathcal{X}^{K^\bullet}:\\
                  N\subset K\in A(S)
                }
      }\psi(1) \psi(g)=\sum_{
                             \substack
                                      { N\subset K\in A(S)
                                       }
                            } \chi^{K^\bullet}(g).$$
Thus,
$$\chi^{{N^\bullet}}(g)=-\sum_{
                               \substack
                                        { N\subset K\in A(S)
                                        }
                            } \chi^{K^\bullet}(g).$$

We can conclude from Cases 1 and 2 that
$$
    =\begin{cases}
           \mathlarger{\sum}_{\psi\in \mathcal{X}^{N^\bullet}}\chi(1)^2  & g\in N,\\
           ~\\
      -\mathlarger{\sum}_{
                          \substack{
                                 N\subset K\in A(S)
                                }
                         } \chi^{K^\bullet}(g)  & g\not\in N.\\
    \end{cases}$$
\end{proof}

\begin{example} The supercharacter table of the normal supercharacter theory generated by all normal pattern subgroups of $UT_4(q)$. Let $t=q-1$, $A=q^4+q^3-2q^2$, and $B=3q^3-4q^2$.

 $$
 {\small {\begin{array}{c|c|c|c|c|c|c|c|c|c|c|c|c|c|c}
  & \begin{tikzpicture}[scale=.15]
 \draw[step=1cm,gray,very thin] (0,0) grid (4,4);
 \end{tikzpicture} &
  \begin{tikzpicture}[scale=.15]
  \draw[step=1cm,gray,very thin] (0,0) grid (4,4);
  \fill[black] (3,3) rectangle (4,4);
  \end{tikzpicture}&
  \begin{tikzpicture}[scale=.15]
   \draw[step=1cm,gray,very thin] (0,0) grid (4,4);
    \fill[black] (2,3) rectangle (4,4);
   \end{tikzpicture}&
  \begin{tikzpicture}[scale=.15]
    \draw[step=1cm,gray,very thin] (0,0) grid (4,4);
     \fill[black] (3,2) rectangle (4,4);
    \end{tikzpicture}&
  \begin{tikzpicture}[scale=.15]
     \draw[step=1cm,gray,very thin] (0,0) grid (4,4);
      \fill[black] (3,1) rectangle (4,4);
     \end{tikzpicture}&
  \begin{tikzpicture}[scale=.15]
\draw[step=1cm,gray,very thin] (0,0) grid (4,4);
 \fill[black] (1,3) rectangle (4,4);
\end{tikzpicture}&
  \begin{tikzpicture}[scale=.15]
\draw[step=1cm,gray,very thin] (0,0) grid (4,4);
 \fill[black] (3,2) rectangle (4,4); \fill[black] (2,3) rectangle (4,4);
\end{tikzpicture}&
  \begin{tikzpicture}[scale=.15]
\draw[step=1cm,gray,very thin] (0,0) grid (4,4);
 \fill[black] (3,2) rectangle (4,4); \fill[black] (1,3) rectangle (4,4);
\end{tikzpicture}&
  \begin{tikzpicture}[scale=.15]
\draw[step=1cm,gray,very thin] (0,0) grid (4,4);
 \fill[black] (2,3) rectangle (4,4); \fill[black] (3,1) rectangle (4,4);
\end{tikzpicture}&
  \begin{tikzpicture}[scale=.15]
\draw[step=1cm,gray,very thin] (0,0) grid (4,4);
 \fill[black] (2,2) rectangle (4,4);
\end{tikzpicture}&
  \begin{tikzpicture}[scale=.15]
\draw[step=1cm,gray,very thin] (0,0) grid (4,4);
 \fill[black] (1,3) rectangle (4,4); \fill[black] (3,1) rectangle (4,4);
\end{tikzpicture}&
  \begin{tikzpicture}[scale=.15]
\draw[step=1cm,gray,very thin](0,0) grid (4,4);
 \fill[black] (1,3) rectangle (4,4); \fill[black] (2,2) rectangle (4,4);
\end{tikzpicture}&
  \begin{tikzpicture}[scale=.15]
\draw[step=1cm,gray,very thin] (0,0) grid (4,4);
 \fill[black] (2,2) rectangle (4,4); \fill[black] (3,1) rectangle (4,4);
\end{tikzpicture}&
  \begin{tikzpicture}[scale=.15]
\draw[step=1cm,gray,very thin](0,0) grid (4,4);
 \fill[black] (2,2) rectangle (4,4); \fill[black] (3,1) rectangle (4,4); \fill[black] (1,3) rectangle (4,4);
\end{tikzpicture}\\
\hline
\begin{tikzpicture}[scale=.15]
\draw[step=1cm,gray,very thin](0,0) grid (4,4);
 \fill[black] (2,2) rectangle (4,4); \fill[black] (3,1) rectangle (4,4); \fill[black] (1,3) rectangle (4,4);
\end{tikzpicture} & 1 & 1& 1& 1& 1& 1& 1& 1& 1& 1& 1& 1& 1& 1\\
\hline
\begin{tikzpicture}[scale=.15]
\draw[step=1cm,gray,very thin] (0,0) grid (4,4);
 \fill[black] (2,2) rectangle (4,4); \fill[black] (3,1) rectangle (4,4);
\end{tikzpicture}&t &t &t &t &t & -1& t& -1& t& t& -1& -1&t& -1\\
\hline
\begin{tikzpicture}[scale=.15]
\draw[step=1cm,gray,very thin](0,0) grid (4,4);
 \fill[black] (1,3) rectangle (4,4); \fill[black] (2,2) rectangle (4,4);
\end{tikzpicture} & t&t&t&t&-1&t&t&t&-1&t&-1&t&-1&-1\\
\hline
  \begin{tikzpicture}[scale=.15]
\draw[step=1cm,gray,very thin] (0,0) grid (4,4);
 \fill[black] (1,3) rectangle (4,4); \fill[black] (3,1) rectangle (4,4);
\end{tikzpicture}& t&t&t&t&t&t&t&t&t&-1&t&-1&-1&-1\\
\hline
 \begin{tikzpicture}[scale=.15]
\draw[step=1cm,gray,very thin] (0,0) grid (4,4);
 \fill[black] (2,2) rectangle (4,4);
\end{tikzpicture}&t^2 &t^2&t^2&t^2&-t&-t&t^2&-t&-t&t^2&1&-t&-t&1\\
\hline
\begin{tikzpicture}[scale=.15]
\draw[step=1cm,gray,very thin] (0,0) grid (4,4);
 \fill[black] (2,3) rectangle (4,4); \fill[black] (3,1) rectangle (4,4);
\end{tikzpicture}& t^2& t^2& t^2& t^2& t^2&-t&t^2&-t&t^2&-t&-t&1&-t&-1\\
\hline
  \begin{tikzpicture}[scale=.15]
\draw[step=1cm,gray,very thin] (0,0) grid (4,4);
 \fill[black] (3,2) rectangle (4,4); \fill[black] (1,3) rectangle (4,4);
\end{tikzpicture}&t^2&t^2&t^2&t^2&-t&t^2&t^2&t^2&-t&-t&-t&-t&1&1\\
\hline
  \begin{tikzpicture}[scale=.15]
\draw[step=1cm,gray,very thin] (0,0) grid (4,4);
 \fill[black] (3,2) rectangle (4,4); \fill[black] (2,3) rectangle (4,4);
\end{tikzpicture}& t^3&t^3&t^3&t^3&-t^2&-t^2&t^3&-t^2&-t^2&-t^2&t&t&t&1\\
\hline
\begin{tikzpicture}[scale=.15]
\draw[step=1cm,gray,very thin] (0,0) grid (4,4);
 \fill[black] (1,3) rectangle (4,4);
\end{tikzpicture}&q^2t&q^2t&q^2t&-q^2&0&q^2t&-q^2&-q^2&0&0&0&0&0&0\\
\hline
 \begin{tikzpicture}[scale=.15]
     \draw[step=1cm,gray,very thin] (0,0) grid (4,4);
      \fill[black] (3,1) rectangle (4,4);
     \end{tikzpicture}&q^2t&q^2t&-q^2&q^2t&q^2t&0&-q^2&0&-q^2&0&0&0&0&2\\
     \hline
     \begin{tikzpicture}[scale=.15]
    \draw[step=1cm,gray,very thin] (0,0) grid (4,4);
     \fill[black] (3,2) rectangle (4,4);
    \end{tikzpicture}& q^2t^2&q^2t^2&q^2t&q^2t^2&-q^2t&0&q^2t&0&q^2&0&0&0&0&-2\\
    \hline
    \begin{tikzpicture}[scale=.15]
   \draw[step=1cm,gray,very thin] (0,0) grid (4,4);
    \fill[black] (2,3) rectangle (4,4);
   \end{tikzpicture}&q^2t^2&q^2t^2&q^2t^2&q^2t&0&-q^2t&q^2t&q^2&0&0&0&0&0&0\\
   \hline
 \begin{tikzpicture}[scale=.15]
  \draw[step=1cm,gray,very thin] (0,0) grid (4,4);
  \fill[black] (3,3) rectangle (4,4);
  \end{tikzpicture}& q^3t^2&q^3t^2&-A&-A&0&0&-B&0&0&0&0&0&0&0\\
  \hline
  \begin{tikzpicture}[scale=.15]
 \draw[step=1cm,gray,very thin] (0,0) grid (4,4);
 \end{tikzpicture}& q^5t&-q^5&0&0&0&0&0&0&0&0&0&0&0&0
 \end{array}}}$$

\end{example}

\section{Primitive Central Idempotents, Faithful Irreducible Characters, and The Finest Normal Supercharacter Theory}\label{finest}
In This section we investigate the connection between primitive central idempotents, faithful irreducible characters, and the finest normal supercharacter theory.

\subsection{Primitive Central Idempotents and The Finest Normal Supercharacter Theory}\label{finest and idempotents}
Let $C_g$ be the conjugacy class of group $G$ containing $g$. For every subset $K\subseteq G$, let $\widehat{K}=\sum_{k\in K} k$. Recall that every character $\chi \in \Irr(G)$ has a corresponding primitive central idempotent $$e_\chi=\frac{\chi(1)}{|G|} \sum_{g\in G}{\chi(g^{-1})g}.$$
These idempotents are orthogonal, i.e. $e_\chi e_\psi=0$ when $\chi,\psi \in \Irr(G)$ and $\chi\neq \psi$. Also recall that $$\widehat{C_g}=\sum_{i}{\frac{|C_g|}{\chi_i (1)}\chi_{i}(g) e_{\chi_i}}.$$ Therefore,

$$|C_g|1-\widehat{C_g}=|C_g|1-\sum_{i}{\frac{|C_g|}{\chi_i (1)}\chi_{i}(g) e_{\chi_i }}=|C_g|\left( 1-\sum_{i}{ \frac{1}{\chi_i (1)}\chi_{i}(g) e_{\chi_i}}\right)=$$

$$|C_g|\left( \sum_{i}{e_{\chi_{i}}}-\sum_{i}{ \frac{1}{\chi_i (1)}\chi_{i}(g) e_{\chi_i }}\right) =|C_g|\left( \sum_{i}{ \left( 1-\frac{1}{\chi_i (1)}\chi_{i}(g)\right)  e_{\chi_i }}\right).$$
Thus,
$$1-\frac{1}{|C_g|}\widehat{C_g}=\sum_{i}{ \left( 1-\frac{\chi_{i}(g)}{{\chi_i (1)}}\right)  e_{\chi_i }}=\sum_{i}{ \left( 1-\frac{\chi_{i}(g)}{{\chi_{i}(1)}}\right)  e_{\chi_i}}.$$
By the last equation i.e., $1-\frac{1}{|C_g|}\widehat{C_g}=\sum_{i}{ \left( 1-\frac{\chi_{i}(g)}{{\chi_{i}(1)}}\right)  e_{\chi_i}}$, we produce a partition of primitive irreducible characters of $G$. Define for an element $g\in G$,
$$E_g:=\left\lbrace e_{\chi_i}: 1-\frac{\chi_{i}(g)}{{\chi_{i}(1)}}\neq 0\right\rbrace  ~\text{and}~ K_g:=\bigcup_{E_g=E_h}{C_h}.$$
Then $\left\lbrace E_g:g\in G\right\rbrace $ is a partition of primitive central idempotents.

Since each irreducible character has a corresponding primitive central idempotent, we can correspond a partition of primitive central idempotents to a given supercharacter theory. But every partition of primitive central idempotents is not corresponded to a supercharacter theory and it is not clear when a partition of primitive central idempotents is related to a supercharacter theory. We show that the partition $\left\lbrace E_g:g\in G\right\rbrace $ is corresponded to the finest normal supercharacter theory which is the one generated by $N(G)$.

\begin{lemma}\label{Kg is a superclass}
Let $G$ be a group and let $S$ be the set of all normal subgroups of $G$. Assume that $E_g=\{e_{\chi_{i}}: i\in I\}$. If
$$N= \bigcap_{\chi \in \Irr(G) \setminus \{\chi_{i}: i\in I\}}\ker \chi,$$
then $K_g= N^{\circ}_{A(S)}$.
\end{lemma}
\begin{proof}
Let $k\in K_g$. Then $E_k=E_g$, and so $k\in \ker \chi$ for every $\chi \in \Irr(G) \setminus \{\chi_{i}: i\in I\}$. Therefore, $k\in N$. Let $H$ be a normal subgroup of $G$ such that $H\subset N$. Then at least there is an irreducible character $\psi\in \Irr(G)$ such that $H\subseteq \ker \psi$, but $N\not \subseteq \ker\psi$. If $k\in H$, then
$$k\in \bigcap_{\chi \in \Irr(G) \setminus \{\chi_{i}: i\in I\}}\ker \chi \cap \ker\psi,$$
and so $E_k\neq E_g$. Thus, $k\not \in K_g$, yielding a contradiction. Therefore, $k$ is in $N$, but $k$ is not in any normal subgroup of $G$ such that $H\subset N$, i.e., $k\in N^\circ.$ So $K_g\subseteq N^\circ$.

Let $h\in N^\circ$. Then $E_h\subseteq E_g$. If $E_g\neq E_h$, there is an irreducible character $\psi\in \{\chi_i: i\in I\}$ such that $h\in \ker \psi$. Let $H=N\cap \ker \psi$. Then $h\not \in N\setminus (N\cap \ker \psi)$. Therefore, $h\not \in N^\circ$, yielding a contradiction. We can conclude that $E_g=E_h$, and so $h\in K_g$. Thus, $N^\circ\subseteq K_g$.
\end{proof}

\begin{theorem}\label{Kg and Ncirc}
Let $G$ be a group and let $S$ be the set of all normal subgroups of $G$.
\begin{enumerate}
\item For every $g\in G$, there is a normal subgroup $N$ of $G$ such that $K_g=N^\circ.$
\item Let $N$ be a normal subgroup of $G$. If $N^\circ_{A(S)}\neq \emptyset$, then for every $g\in N^\circ_{A(S)}$, $K_g=N^\circ_{A(S)}$.
\end{enumerate}
\end{theorem}

\begin{proof}
(1) Let $E_g=\{e_{\chi_{i}}: i \in I\}$ and let
$$N= \bigcap_{\chi \in \Irr(G) \setminus \{\chi_{i}:i\in I\}}\ker\chi.$$
Then by Lemma \ref{Kg is a superclass}, $K_g= N^{\circ}$.

(2) Let $N$ be a normal subgroup of $G$ such that $N^{\circ}\neq \emptyset$. Let $g\in N^{\circ}$. We show that $K_g=N^{\circ}$. Assume that $N=\bigcap_{i\in I}\ker\chi_{i}$.  If there is an irreducible character $\chi\in \Irr(G)\setminus \{\chi_i:i\in I\}$ such that $g\in \ker \chi$, then $g\in H=\bigcap_{i\in I}\ker\chi_{i} \cap \ker\chi$. Thus, $g\in H\subset N$, and so $g\not \in N^\circ$, yielding a contradiction. Therefore, $E_g=\{e_{\chi} : \chi \in \Irr(G) \setminus\{\chi_i: i \in I\}\}$.
By Lemma \ref{Kg is a superclass}, $K_g=N^{\circ}$.
\end{proof}

As a result of Theorem \ref{Kg and Ncirc} we have the following corollary.

\begin{corollary}
Let $G$ be a group. Then the finest normal supercharacter theory of $G$ has
$$\{K_g:g\in G\}$$
as the set of superclasses.
\end{corollary}

Let $G$ be a group. Then for every subset $X$ of $G$ we define the \textbf{\emph{normal closure}} or \textbf{\emph{conjugate closure}} of $X$ by $$X^G:=\langle gxg^{-1}: x\in X, g\in G \rangle=\bigcap_{\substack{N\in N(G):\\X\subseteq N}} N.$$
 When $X=\{g\}$ for some $g\in G$, we denote by $g^G$ the normal closure of $X$. For every element $g\in G$, let $[g]=\{h\in G: g^G=h^G\}$.

\begin{lemma}\label{Ncirc and conjugacy}
Let $G$ be a group and $S$ be the set of all normal subgroups of $G$. Let $N$ be a normal subgroup of $G$. If $N^\circ_{A(S)}\neq \emptyset$, then for every $g\in N^\circ_{A(S)}$, $[g]=N^\circ_{A(S)}$.
\end{lemma}

\begin{proof}
Let $g\in N^\circ\neq \emptyset$. Then $g^G\subseteq N$. Assume that $h\in [g]$ but $h\not\in N^\circ$. Then there exists a normal subgroup $H$ of $G$ such that $h\in H$, and so $h^G\subseteq H$. Thus $g^G=h^G\subseteq H$, i.e, $g\not\in N^\circ$, a contradiction. Therefore, $[g]\subseteq N^\circ$.

We now show that $N^\circ\subseteq [g]$. Let $h\in N^\circ$. If $h^G\neq g^G$, then $h^G$ is a normal subgroup  of $G$ such that $h^G\subset N$. Thus  $h\not\in N^\circ$, a contradiction. Therefore, $N^\circ\subseteq [g]$ for every $g\in N^\circ$.
\end{proof}

\begin{theorem}\label{the finest superclass and closure}
Let $G$ be a group. Then the finest normal supercharacter theory of $G$ has $$\{[g]:g\in G\}$$ as the set of superclasses.
\end{theorem}

\begin{proof}
Note that $\{N^\circ\neq \emptyset: N\in N(G)\}$ is a partition of $G$ and also by Lemma \ref{Ncirc and conjugacy} every  $N^\circ\neq \emptyset$ is equal to $[g]$ for some $g\in G$. We can conclude that $$\{[g]:g\in G\}=\{N^\circ\neq \emptyset: N\in N(G)\},$$
which is the set of superclasses of the finest normal supercharacter theory of $G$.
\end{proof}

\subsection{Faithful Irreducible Characters and The Finest Normal Supercharacter Theory}
 A character $\psi$ of a group $G$ is called faithful if ${\ker} ~\psi$ is trivial. In this subsection, we investigate the connection between faithful irreducible characters and the finest normal supercharacter theory.

For a normal subgroup $N$ of $G$, define
$$\cX_{N}:=\{\psi\in {\rm Irr}(G): \overline{\psi} ~\text{is a faithful irreducible character of }~ G/N~\text{for some}~ N\lhd G\}$$
and
$$\chi_N=\sum_{\psi\in \cX_N}\psi(1)\psi.$$

\begin{proposition}\label{Nfinest_supercharacter}
Let $G$ be a group. Then the set of supercharacters for the finest normal supercharacter theory is
$$\{\chi_N\neq 0: N\in N(G)\}.$$
\end{proposition}

\begin{proof}
 It is enough to show that
$$\{\cX_N\neq \emptyset: N\in N(G)\}=\{\cX^{N^\bullet}\neq \emptyset: N\in N(G)\}.$$
 For an arbitrary normal subgroup of $N$, let $\psi \in \cX^{N^\bullet}$. By the definition of $\cX^{N^\bullet}$, we have $N\subseteq \ker\psi$ but $K\not\subseteq \ker\psi$ for any normal subgroup $K$ of $G$ with $N\subset K$. Since $\ker\psi$ is a normal subgroup, we can conclude that $\ker \psi=N$. Therefore, we have $\overline{\psi}$ is a faithful irreducible character of $G/N$, and so $\psi \in \cX_N$. Now, let $\psi \in \cX_N$. Then $N=\ker\psi$, and so $\psi \in \cX^{N^\bullet}$. We can conclude that $\cX_N=\mathcal{X}^{N^\bullet}$.
\end{proof}

\subsection{Constructing the Supercharacter Table of the Finest Supercharacter Theory From the Character Table}

For a $n\times n$ matrix $T=(t_{ij})$ define
$$T_j=\{(k,j): t_{kj}=1 \}~~\text{and}~~T^{i}=\{(i,k): t_{ik}=1\}.$$
By the following steps we construct the finest supercharacter theory of a given group $G$ by its character table.

\begin{enumerate}
\item Divide each row $i$ of the character table $T$ of $G$ by $\chi_i(1)$. Denote the new table by $A$.
\item  Rearrange the columns of $A$ such that two columns $i$ and $j$ are consecutive if $A_i=A_j$. Denote the new table by $B$.
\item Draw some vertical lines in $B$ to distinguish the classification of conjugacy classes with the same $A_i$. This lines make a partition of $G$, and this partition is the set of superclasses of the finest normal supercharacter theory of $G$.
\item  Rearrange the rows of $B$ such that two rows $i$ and $j$ are consecutive if $B^i=B^j$. Denote the new table by $C$.
\item Draw some horizontal lines in $C$ to distinguish the classification of irreducible characters with the same $B_i$. This lines make a partition of ${\rm Irr}(G)$. This partition is the partition of irreducible characters of the finest normal supercharacter theory of $G$.
\end{enumerate}

Let $G=C_3\times C_3=\langle a\rangle \times \langle b \rangle$. We construct the supercharacter table for the finest normal supercharacter theory of $G$.
$$
\left.
{\small \begin{array}{c|cccccccccc}
      &   1         &      b       &       b^2    &      a       &     ab      &     ab^2    &      a^2   &   a^2b      &   a^2b^2\\
      \hline
{\one}&       1     &      1       &      1       &      1       &       1     &     1       &      1     &      1      &     1\\
\chi_2&       1     &     1        &        1     & -\zeta_3 - 1 &-\zeta_3 - 1 &-\zeta_3 - 1 &    \zeta_3 &     \zeta_3 &     \zeta_3\\
\chi_3&        1    &      1       &       1      &    \zeta_3   &  \zeta_3    &   \zeta_3   &-\zeta_3 - 1& -\zeta_3 - 1& -\zeta_3 - 1\\
\chi_4&         1   & -\zeta_3 - 1 &     \zeta_3  &         1    & -\zeta_3 - 1&   \zeta_3   &         1  &-\zeta_3 - 1 & \zeta_3 \\
\chi_5&      1      &    \zeta_3   &-\zeta_3 - 1  &        1     &   \zeta_3   &-\zeta_3 - 1 &       1    &    \zeta_3  & -\zeta_3 - 1\\
\chi_6&    1        &-\zeta_3 - 1  &    \zeta_3   & -\zeta_3 - 1 &     \zeta_3 &         1   & \zeta_3    &       1     & -\zeta_3 - 1\\
\chi_7&       1     &   \zeta_3    & -\zeta_3 - 1 &     \zeta_3  & -\zeta_3 - 1&        1    &-\zeta_3 - 1&         1   &   \zeta_3\\
\chi_8&       1     &-\zeta_3- 1   &     \zeta_3  &     \zeta_3  &        1    & -\zeta_3 - 1& -\zeta_3- 1&    \zeta_3  &       1\\
\chi_9&    1        &  \zeta_3     & -\zeta_3 - 1 & -\zeta_3 - 1 &       1     &     \zeta_3 &    \zeta_3 &-\zeta_3 - 1 &       1\\
\end{array}}
\right.
$$
$$\Downarrow$$
$$
\left.
{\small \begin{array}{c|c|cc|cc|cc|ccc}
      &   1         &      b       &       b^2    &      a       &     a^2     &     ab^2   &   a^2b          &      ab    &   a^2b^2      \\
      \hline
{\one}&       1     &      1       &      1       &      1       &       1     &     1         &     1        &      1    &      1     \\
\chi_2&       1     &     1        &        1     & -\zeta_3 - 1 &    \zeta_3  &-\zeta_3 - 1   &     \zeta_3         & \zeta_3 &    -\zeta_3-1\\
\chi_3&        1    &      1       &       1      &    \zeta_3   &  -\zeta_3-1 &   \zeta_3      & -\zeta_3 - 1        &-\zeta_3-1  & \zeta_3 \\
\chi_4&         1   & -\zeta_3 - 1 &     \zeta_3  &         1    & 1           &   \zeta_3      & -\zeta_3-1       &-\zeta_3 - 1&    \zeta_3   \\
\chi_5&      1      &    \zeta_3   &-\zeta_3 - 1  &        1     &   1         &-\zeta_3 - 1    & \zeta_3 - 1        &   \zeta_3  & -\zeta_3-1  \\
\chi_6&    1        &-\zeta_3 - 1  &    \zeta_3   & -\zeta_3 - 1 &     \zeta_3 &         1       & 1         & \zeta_3    & -\zeta_3 - 1   \\
\chi_7&       1     &   \zeta_3    & -\zeta_3 - 1 &     \zeta_3  & -\zeta_3 - 1&        1       &   1          &-\zeta_3 - 1&   \zeta_3    \\
\chi_8&       1     &-\zeta_3- 1   &     \zeta_3  &     \zeta_3  &-\zeta_3-1   & -\zeta_3 - 1    &        \zeta_3         & 1          &   1 \\
\chi_9&    1        &  \zeta_3     & -\zeta_3 - 1 & -\zeta_3 - 1 &  \zeta_3    &     \zeta_3      &           -\zeta_3 - 1       &  1         & 1 \\
\end{array}}
\right.
$$
$$\Downarrow$$
$$
\left.
{\small \begin{array}{c|c|cc|cc|cc|ccc}
      &   1         &      b       &       b^2    &      a       &     a^2     &     ab^2   &   a^2b          &      ab    &   a^2b^2      \\
      \hline
{\one}&       1     &      1       &      1       &      1       &       1     &     1         &     1        &      1    &      1     \\
\hline
\chi_2&       1     &     1        &        1     & -\zeta_3 - 1 &    \zeta_3  &-\zeta_3 - 1   &     \zeta_3         & \zeta_3 &    -\zeta_3-1\\
\chi_3&        1    &      1       &       1      &    \zeta_3   &  -\zeta_3-1 &   \zeta_3      & -\zeta_3 - 1        &-\zeta_3-1  & \zeta_3 \\
\hline
\chi_4&         1   & -\zeta_3 - 1 &     \zeta_3  &         1    & 1           &   \zeta_3      & -\zeta_3-1       &-\zeta_3 - 1&    \zeta_3   \\
\chi_5&      1      &    \zeta_3   &-\zeta_3 - 1  &        1     &   1         &-\zeta_3 - 1    & \zeta_3 - 1        &   \zeta_3  & -\zeta_3-1  \\
\hline
\chi_6&    1        &-\zeta_3 - 1  &    \zeta_3   & -\zeta_3 - 1 &     \zeta_3 &         1       & 1         & \zeta_3    & -\zeta_3 - 1   \\
\chi_7&       1     &   \zeta_3    & -\zeta_3 - 1 &     \zeta_3  & -\zeta_3 - 1&        1       &   1          &-\zeta_3 - 1&   \zeta_3    \\
\hline
\chi_8&       1     &-\zeta_3- 1   &     \zeta_3  &     \zeta_3  &-\zeta_3-1   & -\zeta_3 - 1    &        \zeta_3         & 1          &   1 \\
\chi_9&    1        &  \zeta_3     & -\zeta_3 - 1 & -\zeta_3 - 1 &  \zeta_3    &     \zeta_3      &           -\zeta_3 - 1       &  1         & 1 \\
\end{array}}
\right.
$$
$$\Downarrow$$
The set of supercharacters of the finest normal supercharacter theory is $$ \{ \{{\one}\},\{\chi_2,\chi_3\},\{\chi_4,\chi_5\},\{\chi_6,\chi_7\},\{\chi_8,\chi_9\}  \}$$
The set of superclasses of the finest normal supercharacter theory is $$ \{\{1\},\{a,a^2\},\{b,b^2\},\{ab^2,a^2b\},\{ab,a^2b^2\}\}.$$
$$\Downarrow$$
Here is the supercharacter table for the finest normal supercharacter theory of $G$.
$$
\left.
{\small \begin{array}{c|ccccc}
                                 &   \{1\}         &      \{b ,b^2\}    &      \{a,a^2\}          &     \{ab^2, a^2b\}  & \{ab ,a^2b^2\}\\
      \hline
{\one}                          &       1     &      1             &      1                  &     1               &        1                      \\
\chi_2+\chi_3                    &   2         &     2              & -1                      &     -1               &        -1                          \\
\chi_4+\chi_5                    & 2           &     -1             &         2               &   -1                  &       -1                       \\
\chi_6+\chi_7                   &  2         &      -1            & -1                      &         2              &         -1         \\
\chi_8+\chi_9                    &  2         &      -1            & -1                      &         -1              &        2          \\

\end{array}}
\right.
$$

\section{$NSup(G)$ is not a subset of the union of  $AutSup(G)$, $ACSup(G)$, and $Sup^{*}(G)$}\label{new}

\hspace{4mm} In the following example we show that $Sup^{*}(G)\cap AutSup(G)$ is not a subset of $NSup(G)$ and there is a normal supercharacter theory which is not in the union of $AutSup(G)$, $ACSup(G)$, and $Sup^{*}(G)$.\\

\noindent{\bf Example.} Let $G=C_3 \times C_4=\langle g \rangle \times \langle h \rangle$. Note that the supercharacter theory correspond to superclass theory $\{ C_g: g\in G\}$  is in $Sup^{*}(G)\cap AutSup(G)$, but it is not in $NSup(G)$. Therefore, $Sup^{*}(G)\cap AutSup(G)$ is not a subset of $NSup(G)$.   We now construct the normal supercharacter theory generated by $S=\{\langle g \rangle \times 1,1\times \langle h \rangle \}$ and we show that this supercharacter theory is not  in the union of $AutSup(G)$, $ACSup(G)$, and $Sup^{*}(G)$.

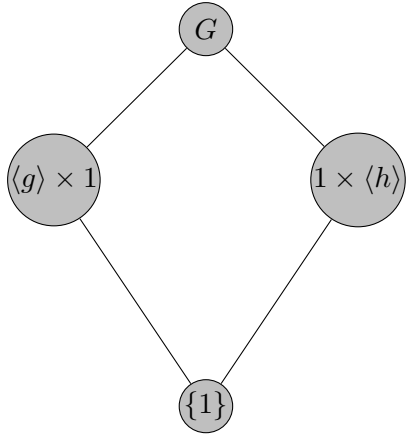
\begin{center}
\begin{tikzpicture}[scale=1, auto,swap]
\tikzstyle{every node}=[circle,fill=black!25,minimum size=20pt,inner sep=0pt];
\node (a) at (0,1)[draw, circle, fill][label=right:${}$]{$ \{1\}$};
\node (b) at (-2,4)[draw, circle, fill][label=left:${}$]{$ \langle g \rangle\times 1$};
\node (c) at (2,4)[draw, circle, fill][label=right:${}$]{$1\times \langle h \rangle$};
\node (d) at (0,6)[draw, circle, fill][label=right:${}$]{$G$};

\draw (a)--(b)(c)--(d)(b)--(d)(a)--(c);
																		
\end{tikzpicture}

\end{center}
The set of superclasses for the normal supercharacter theory generated by  $S$ is $\{\{1\}, \{(g,1),$\\$(g^2,1)\},\{(1,h),(1,h^2),(1,h^3)\},\{(g,h),(g,h^2), (g,h^3),(g^2,h),(g^2,h^2),(g^2,h^3)\}\}.$

 Let $(\mathcal{X},\mathcal{K})\in AutSup(G)$. Since $Aut(G)\cong \Bbb{Z}_2 \times \Bbb{Z}_2$, every orbit has at most $4$ members. Note that the members of $\mathcal{K}$ are the unions of the $A$-orbits on the classes of $G$. Therefore, every members of $\mathcal{K}$ has at most cardinality $4$. But we have a superclass with cardinality $6$ in the normal  supercharacter theory generated by $S$. Thus, this normal supercharacter theory for $G$ is not in $AutSup(G)$.
\\

Since $|Aut(Q[\zeta_{|G|}:Q])|=4$, the largest superclass in $\mathcal{K}\in ACSup(G)$ has cardinality $4$. Note that the set of superclasses for the normal supercharacter theory generated by  $S$ has a superclass of cardinality $6$. Therefore, the normal supercharacter theory generated by $S$ is not in  $ACSup(G)$.\\

 Now we show that normal supercharacter theory generated by  $S$ is not in $Sup^{*}(G)$. If we choose a subgroup of order $2$ and construct the $\ast$-product, then there are two superclasses with cardinality 1, but we only have one superclass with cardinality $1$ in the normal supercharacter theory generated by  $S$. Let us choose a subgroup of order $4$. Then $\{(g,1),(g^2,1)\}$ is not a superclass of this supercharacter theory.
Now we choose a subgroup of order 3, and construct the supercharacter theory by $\ast$-product. Then $\{(1,h),(1,h^2),(1,h^3)\}$ is not a superclass of this supercharacter theory.
\\

Therefore, the normal supercharacter theory generated by  $S$ is not in $AutSup(G)\cup ACSup(G)\cup Sup^{*}(G)$.

Farid Aliniaeifard, \textsc{Department of Mathematics and Statistics, York University, Toronto, ON M3J 1L2 Canada}\\
 \textit{E-mail address:} \texttt{faridanf@mathstat.yorku.ca}

\end{document}